\documentclass[12pt]{amsart}

\usepackage{amsthm}
\usepackage{amsmath}
\usepackage{amsfonts}
\usepackage{amssymb}
\usepackage{tikz-cd}
\usepackage{graphicx}
\usepackage{float}

\usepackage{mathrsfs}

\def\int{{\sf int}}
\newcommand\cl{{\sf cl}}

\newcommand\up{{\uparrow}}
\newcommand\down{{\downarrow}}
\newcommand\Op{\mathcal{O}}

\newcommand{\T}{\mathscr T_2}

\newtheorem{theorem}{Theorem}[section]
\newtheorem{lemma}[theorem]{Lemma}

\newtheorem{corollary}[theorem]{Corollary}

\theoremstyle{definition}
\newtheorem{definition}[theorem]{Definition}

\newtheorem{remark}[theorem]{Remark}
\newtheorem{convention}[theorem]{Convention}

\begin{document}

\title{The Frame of Nuclei of an Alexandroff Space}
\author{F. \'{A}vila, G.~Bezhanishvili, P.~J.~Morandi, A.~Zald{\'i}var}
\date{}


\begin{abstract}
Let $\Op S$ be the frame of open sets of a topological space $S$, and let $N(\Op S)$ be the frame of nuclei of $\Op S$. For an Alexandroff space $S$, we prove that $N(\Op S)$ is spatial iff the infinite binary tree $\T$ does not embed isomorphically into $(S, \le)$, where $\le$ is the
specialization preorder of $S$.
\end{abstract}

\maketitle

\section{Introduction}

Nuclei play an important role in pointfree topology as they characterize homomorphic images of frames (or dually sublocales of locales). For a frame
$L$, let $N(L)$ be the frame of nuclei of $L$, also known as the assembly of $L$. The frame $N(L)$ has been investigated by many authors; see, e.g.,
\cite{DP66, Isb72, Sim78, BM79, Sim80, Mac81, Joh82, NR87, Isb91, Wil94, Ple02, BG07, BGJ13, Sim14, BGJ16} (which are listed in chronological order).
For example, Beazer and Macnab \cite{BM79} gave a necessary and sufficient condition for $N(L)$ to be boolean; Niefield and Rosenthal \cite{NR87} gave
necessary and sufficient conditions for $N(L)$ to be spatial, and derived that if $N(L)$ is spatial, then so is $L$; Simmons \cite{Sim80} proved that
if $L$ is the frame of opens of a $T_0$-space $S$, then $N(L)$ is boolean iff $S$ is scattered; and Isbell \cite{Isb91} proved that if $L$ is the
frame of opens of a sober space $S$, then $N(L)$ is spatial iff $S$ is weakly scattered (see Section~2 for definitions).

In \cite{BG07} the study of $N(L)$ using the spectrum of $L$ was initiated. We utilized this approach in \cite{ABMZ18a} to generalize the results
mentioned above (and also to give alternate proofs of these results). One of the main results of \cite{ABMZ18a} gives a necessary and sufficient
condition for $N(L)$ to be spatial in terms of the spectrum of $L$, from which it is derived that if $L=\Op S$ is the frame of opens of a topological
space $S$, then $N(L)$ is spatial iff the soberification of $S$ is weakly scattered.

In the present paper we restrict our attention to Alexandroff
spaces (in which each point has a least neighborhood). It is well known that Alexandroff spaces correspond to preordered sets, and Alexandroff
$T_0$-spaces to partially ordered sets. Thus, the frame of opens of an Alexandroff space $S$ is isomorphic to the frame of upward closed sets of
a preordered set. We prove that for an Alexandroff $T_0$-space $S$, the frame $N(\Op S)$ is spatial iff the infinite binary tree $\T$ is not
isomorphic to a subposet of $S$. From this we derive that for an arbitrary Alexandroff space $S$, the frame $N(\Op S)$ is spatial iff the infinite
binary tree $\T$ does not embed isomorphically into $S$.

%If $L$ is such a frame, our main result gives a necessary and sufficient condition for $N(L)$ to be spatial internally in terms
%of the order structure of $S$. This condition is simpler to verify than the condition for an arbitrary topological space given in \cite{ABMZ18a}
%which is in terms of the soberification of $S$.

We point out that if $S$ is a poset, then Simmons's characterization of when $N(\Op S)$ is boolean
%Our main result has a number of consequences. For a partially ordered set $S$, the result of Simmons mentioned above
takes on the following form:
%If $L$ is the frame of opens of a poset $S$, then
$N(\Op S)$ is boolean iff $S$ is noetherian (has no infinite ascending chains). Since $S$ being
noetherian is equivalent to $S$ being sober,
%the result of
Isbell's characterization of when $N(\Op S)$ is spatial for sober $S$
%mentioned above
does not yield any examples of posets $S$ such that $N(\Op S)$ is
spatial but not boolean. Our main result yields many such examples. Indeed, it implies that if $S$ is a poset with no infinite antichains,
then $N(\Op S)$ is spatial. In particular, if $S$ is totally ordered, then $N(\Op S)$ is spatial.
%We prove that if $S$ is a totally ordered set, then $N(L)$ is always spatial.
Thus, each totally ordered set (or more generally a poset with no infinite antichains) that
is not noetherian yields an example of a spatial $N(\Op S)$ which is not boolean.
%More generally, we prove that if $S$ is a thin poset (has no
%infinite anctichains), then $N(L)$ is spatial. Finally, if $S$ is a tree, then we prove that $N(L)$ is spatial iff the infinite binary tree
%$\T$ does not embed in $S$.

\section{Preliminaries}

\begin{definition}
For a frame $L$, let $X_L$ be the set of prime filters of $L$. We will refer to $X_L$ as the \emph{spectrum} of $L$.
\end{definition}
If $\le_L$ is the inclusion order, then $(X_L,\le_L)$ is a poset (partially ordered set). For $a\in L$, let
\[
\eta(a)=\{x\in X_L\mid a\in x\}.
\]
There are several topologies on $X_L$, two of which play an important role in our considerations. Define $\tau_L$ and $\pi_L$ on $X_L$ by letting
\[
\{\eta(a)\mid a\in L\} \ \mbox{ and } \ \{\eta(a)\setminus\eta(b)\mid a,b\in L\}
\]
be the bases for $\tau_L$ and $\pi_L$, respectively. It is well known that $\tau_L$ is a spectral topology (sober and coherent) and $\pi_L$ is the
patch topology of $\tau_L$, hence $\pi_L$ is a Stone topology (compact, Hausdorff, zero-dimensional).

The ordered space $(X_L,\pi_L,\le_L)$ is a \emph{Priestley space}; that is, a compact ordered space satisfying the \emph{Priestley separation axiom}:
$x\not\le_L y$ implies there is a clopen upset containing $x$ and missing $y$. When there is no danger of confusion, we will abbreviate
$(X_L,\pi_L,\le_L)$ by $X_L$. Since $L$ is a Heyting algebra, $X_L$ is in fact an \emph{Esakia space} (the downset of clopen is clopen).
In addition, since $L$ is complete, the closure of each open upset is a clopen upset. Such spaces are often referred to as
\emph{extremally order-disconnected Esakia spaces} (see, e.g., \cite[Sec.~3]{ABMZ18a} and the references therein).

For $A\subseteq X_L$ we recall that the upset ${\uparrow}A$ and the downset ${\downarrow}A$ are defined by
\begin{align*}
{\uparrow}A=\{x\in X_L\mid a\le x \mbox{ for some } a\in A\}, \\
{\downarrow}A=\{x\in X_L\mid x\le a \mbox{ for some } a\in A\}.
\end{align*}
It is well known that if $A$ is closed, then both ${\uparrow}A$ and ${\downarrow}A$ are closed.

The next definition originates in \cite{BG07}. The current terminology was given in \cite[Def.~4.1]{ABMZ18a}.
\begin{definition}
Let $L$ be a frame and $X_L$ its spectrum.
\begin{enumerate}
\item We call a closed subset $F$ of $X_L$ \emph{nuclear} provided ${\downarrow}(F\cap U)$ is clopen for each clopen $U$ of $X_L$.
\item Let $N(X_L)$
be the set of nuclear subsets of $X_L$.
\item If $F=\{x\}$ is nuclear, then we call $x$ a \emph{nuclear point}.
\item Let $Y_L$ be the subset of $X_L$ consisting of nuclear points of $X_L$.
\end{enumerate}
\end{definition}

\begin{theorem}
\cite[Thm.~30]{BG07}  Let $L$ be a frame and $X_L$ its spectrum. Then
%\begin{enumerate}
%\item
$N(L)$ is dually isomorphic to $N(X_L)$.
%(see \cite[Def.~5.1]{ABMZ18a}).
%\item \cite[Thm.~5.5]{ABMZ18a} $L$ is spatial iff $Y_L$ is dense in $(X_L,\pi_L)$.
%\end{enumerate}
\end{theorem}

 We denote the restrictions of $\tau_L$ and $\pi_L$ to $Y_L$ by $\tau$ and $\pi$, respectively. Let $\Op_\tau(Y_L)$ be the frame of opens of
 $(Y_L,\tau)$ and $\Op_\pi(Y_L)$ the frame of opens of $(Y_L,\pi)$.

\begin{theorem}\cite[Thm.~5.9]{ABMZ18a} \label{thm: spatial}
For a frame $L$, the following are equivalent.
\begin{enumerate}
\item The frame $N(L)$ is spatial.
\item If $N \in N(X_L)$ is nonempty, then $N \cap Y_L \ne \varnothing$.
\item $N(L)$ is isomorphic to $\Op_\pi(Y_L)$.
\end{enumerate}
\end{theorem}

For $F$ a closed subset of $X_L$, let $\max F$ be the set of maximal points and $\min F$ the set of minimal points of $F$. It is well known that
for each $x\in F$ there are $m\in\min F$ and $M\in\max F$ with $m\le_L x\le_L M$. Therefore, if $F\ne\varnothing$, then $\max F,\min F\ne\varnothing$.
The following is a useful corollary of Theorem~\ref{thm: spatial}.

\begin{corollary} \label{cor: spatial}
Let $L$ be a frame. Then $N(L)$ is spatial iff $\max U \cap Y_L\ne \varnothing$ for each nonempty clopen downset $U$ of $X_L$.
\end{corollary}

\begin{proof}
First suppose that $N(L)$ is spatial. Let $U$ be a nonempty clopen downset of $X_L$. Then $U\in N(X_L)$, so $\max U\in N(X_L)$ by
\cite[Cor.~4.5]{ABMZ18a}. Since $U\ne\varnothing$, we have $\max U\ne\varnothing$. Thus, $\max U \cap Y_L \ne \varnothing$ by
Theorem~\ref{thm: spatial}.

Conversely, suppose that $\max U \cap Y_L\ne \varnothing$ for each clopen downset $U$ of $X_L$. Let $N \in N(L)$ be nonempty. Then $\max N$
is nonempty and $\max N \in N(L)$. If $U = \down N$, then $U$ is a clopen downset and $\max U = \max N$. Therefore, $\max U \in N(L)$ is
nonempty. By our assumption, $\max U \cap Y_L \ne \varnothing$. This implies that $\max N \cap Y_L \ne \varnothing$. Thus, $N(L)$ is spatial
by Theorem~\ref{thm: spatial}.
\end{proof}

 Let $S$ be a topological space and $T$ a subspace of $S$. We recall that $x\in T$ is an \emph{isolated} point of $T$ if
$\{x\}=U\cap T$ for some open subset $U$ of $S$, and that $x$ is a \emph{weakly isolated} point of $T$ if
$x\in U\cap T\subseteq\overline{\{x\}}$ for some open subset $U$ of $S$. Then $X$ is \emph{scattered} if each nonempty closed subspace
of $X$ has an isolated point, and $X$ is \emph{weakly scattered} if each nonempty closed subspace of $X$ has a weakly isolated point.
For a spatial frame $L$, to the conditions of Theorem~\ref{thm: spatial} and Corollary~\ref{cor: spatial}, we could add that $(Y_L,\tau)$ is weakly scattered.

\begin{theorem}\cite[Thm.~7.3]{ABMZ18a} \label{thm: weakly scattered}
Let $L$ be a spatial frame. Then $N(L)$ is spatial iff $(Y_L,\tau)$ is weakly scattered.
\end{theorem}

\begin{remark}
If in Theorem~\ref{thm: weakly scattered} we do not assume that $L$ is spatial, then to $(Y_L, \tau)$ being weakly scattered we need to add the condition that $Y_L$ is dense in $(X_L, \pi_L)$ \cite[Thm.~5.5]{ABMZ18a}.
\end{remark}

Since $N(L)$ spatial implies that $L$ is spatial, from now on we will assume that $L$ is a spatial frame, so $L=\Op S$ for some
topological space $S$. There is a natural map $\varepsilon:S\to X_L$ given by
\[
\varepsilon(s)=\{U\in\Op S\mid s\in U\}.
\]
For $U\in\Op S$ we have that $\varepsilon^{-1}\eta(U)=U$. Therefore, $\varepsilon$ is a continuous map from $S$ to $(X_L,\tau_L)$,
and it is an embedding iff $S$ is a $T_0$-space.

\begin{theorem}\label{thm: Y}
\cite[Prop.~7.1]{ABMZ18a} $(Y_L,\tau)$ is homeomorphic to the soberification of $S$.
\end{theorem}
We can view $\varepsilon$ as the soberification map from $S$ to $Y_L$.

\begin{remark} \label{rem: S in Y}
For the reader's convenience, we give an elementary argument for why $\varepsilon[S] \subseteq Y_L$. To see this we must show that
$\down \varepsilon(s)$ is clopen in $(X_L,\pi_L)$. It is sufficient to observe that
$\down \varepsilon(s) = X_L\setminus\eta\left(S \setminus \overline{\{s\}}\right)$ for each $s \in S$. We have
\begin{eqnarray*}
x \in \down \varepsilon(s) & \Longleftrightarrow & x \le \varepsilon(s) \\
& \Longleftrightarrow & (\forall U\in\Op S)(U\in x \Rightarrow s\in U) \\
& \Longleftrightarrow & (\forall U\in\Op S)(s\notin U \Rightarrow U\notin x).
\end{eqnarray*}
On the other hand,
\[
x \in X_L\setminus\eta\left(S \setminus \overline{\{s\}}\right) \Longleftrightarrow S \setminus \overline{\{s\}}\notin x.
\]
Since $S \setminus \overline{\{s\}}$ is the largest open set missing $s$, we conclude that
\[
x \in \down \varepsilon(s) \mbox{ iff } x\in X_L\setminus\eta\left(S \setminus \overline{\{s\}}\right),
\]
yielding the desired equality.
\end{remark}

One of the key techniques of Simmons in the study of $N(\Op S)$ is the notion of the front topology on $S$. We recall that the \emph{front
topology} on $S$ is the topology $\tau_F$ generated by $\{U\setminus V\mid U,V\in\Op S\}$.

\begin{theorem}\cite[Lem.~7.9]{ABMZ18a} \label{thm: compactification}
The map $\varepsilon:(S,\tau_F)\to(X_L,\pi_L)$ is a compactification of $(S,\tau_F)$.
\end{theorem}

Since $(X_L,\pi_L)$ is a Stone space, it follows that $(Y_L, \pi)$ is a zero-dimensional Hausdorff space. But in general,
$(Y_L, \pi)$ is not compact.

%\section{A characterization of spatiality of $N(\Op S)$ for $S$ an Alexandroff space}
\section{Main Theorem}

We recall that $S$ is an \emph{Alexandroff space} if the intersection of an arbitrary family of open sets is open. Equivalently $S$ is
Alexandroff iff each point of $S$ has a least open neighborhood. It is well known that Alexandroff spaces are in 1-1 correspondence with
preordered sets. Indeed, the \emph{specialization preorder} on $S$, defined by $s\le t$ iff $s\in\overline{\{t\}}$, is reflexive
and transitive, and $U$ is open in $S$ iff $U$ is an upset (that is, $s\in U$ and $s\le t$ imply $t\in U$; equivalently $\up U=U$).
Moreover, $S$ is $T_0$ iff the specialization order is a partial order. From now on we will think of Alexandroff spaces as preorders
$(S,\le)$ and of the frame $\Op S$ as the frame of upsets of $(S,\le)$. Then closed sets are downsets ($s\le t$ and $t\in F$ imply $s\in F$
or equivalently $\down F=F$) and the closure of $A\subseteq S$ is $\down A$.

For a preorder $(S,\le)$ define an equivalence relation $\sim$ on $S$ by $x\sim y$ iff $x\le y$ and $y\le x$. Then $(S_0,\le_0)$ is a partial
order, known as the \emph{skeleton} of $(S,\le)$, where $S_0={S/{\sim}}$ and $[x]\le_0[y]$ iff $x\le y$. Topologically, the skeleton $S_0$ is the $T_0$-reflection
of $S$. Since $\Op S$ is isomorphic to $\Op S_0$, we may restrict our attention to posets.

Let $S$ be a poset and let $L=\Op S$. The spectrum $X_L$ of $L$ was described in \cite[Sec.~3]{BGMM06} as the Nachbin compactification of $S$.
We recall that an ordered topological space $(X,\tau,\le)$ is a \emph{Nachbin space} if $X$ is compact (Hausdorff) and $\le$ is closed in the
product topology, and that an \emph{order-compactification} of an ordered topological space $(X,\tau,\le)$ is a Nachbin space $(Y,\pi,\le)$
such that there is a topological and order embedding $e:X\to Y$ with $e[X]$ topologically dense in $Y$. A \emph{Nachbin compactification} of $(X,\tau,\le)$
is then the largest order-compactification of $(X,\tau,\le)$. It is an order-topological analogue of the Stone-\v{C}ech compactification.
In particular, every order-preserving continuous map from $X$ to a Nachbin space has a unique extension to the Nachbin compactification of $X$.
Viewing a poset $S$ as an ordered topological space with the discrete topology, we have the following:

\begin{theorem} \cite[Prop.~3.4]{BGMM06} \label{thm: Nachbin}
Let $S$ be a poset and $L=\Op S$ the frame of upsets of $S$. The Nachbin compactification of $S$ is order-homeomorphic to $(X_L,\pi_L,\le_L)$.
\end{theorem}

\begin{convention}\label{convention}
To simplify notation, from now on we will drop the subscript from $(X_L,\pi_L,\le_L)$ and simply write $(X,\pi,\le)$. We will also abbreviate
$\tau_L$ by $\tau$. Similarly, we will write $Y$ instead of $Y_L$, so $(Y,\tau)$ is a subspace of $(X,\tau)$ and $(Y,\pi)$ is a subspace of
$(X,\pi)$. We will write $\cl$ for the closure in $(X,\pi)$. Since $\varepsilon : S \to (Y, \tau)$ is the soberification of $S$, we identify
$S$ with its image $\varepsilon[S]$ in $Y$, and view $S$ a subspace of $(Y, \tau)$.
\end{convention}

\begin{lemma} \label{lem: downsets in S}
Let $T$ be a subset of $S$.
\begin{enumerate}
\item If $A$ is an upset of $T$, then there is a clopen upset $U$ of $X$ with $A = U \cap T$.
\item If $B$ is a downset of $T$, then there is a clopen downset $V$ of $X$ with $B = V \cap T$.
\end{enumerate}
\end{lemma}

\begin{proof}
(1) Since $A$ is an upset of $T$, there is an upset $A'$ of $S$ with $A = A' \cap T$. Because $S$ is a subspace of $(Y, \tau)$, there is
an open subset $U'$ of $(Y, \tau)$ with $A' = U' \cap S$. By \cite[Lem.~5.3(1)]{ABMZ18a}, there is a clopen upset $U$ of $X$ with
$U' = U \cap Y$. Therefore, $A = U \cap T$.

(2) The proof is similar to (1), but uses \cite[Lem.~5.3(2)]{ABMZ18a}.
\end{proof}

Since the topology on $S$ is discrete, the next lemma can be thought of as an order-theoretic analogue of  \cite[Cor.~3.6.4]{Eng89}.

\begin{lemma} \label{lem: closure of a downset} \label{lem: max(D)} \label{lem: disjoint}
Let $D$ be a downset of $S$.
%\begin{enumerate}
%\item $\cl(D)$ is a clopen downset of $X$.
%\item $\max D \subseteq \max\cl(D)$.
If $A,B \subseteq D$ with $A \cap B = \varnothing$ and $B$ an upset of $D$, then $\cl(A) \cap \cl(B) = \varnothing$.
%\end{enumerate}
\end{lemma}

\begin{proof}
%(1) Since $D$ is a downset of $S$, by Lemma~\ref{lem: downsets in S}(2) there is a clopen downset $V$ of $X$ with $D = V \cap S$.
%By Theorem~\ref{thm: compactification}, $S$ is dense in $X$. Therefore, $\cl(D) = \cl(V \cap S) = V \cap \cl(S) = V$ as $V$ is clopen.
%Thus, $\cl(D)$ is a clopen downset of $X$.
%
%(2) Let $s \in \max D$. Then $s\in\cl(D)$. If $s\notin\max\cl(D)$, then there is $x \in \cl(D)$ with $s < x$. By the Priestley
%separation axiom, there is a clopen upset $U$ of $X$ with $x \in U$ and $s \notin U$. Since $x \in \cl(D)$, we have
%$U \cap D \ne \varnothing$. Now, $\up s \cap S$ is an open set of $S$.
%Therefore, by Lemma~\ref{lem: downsets in S}(1), there is a clopen upset $V$ of $X$ with $\up s \cap S = V \cap S$.
%Since $\up s \cap D = \{s\}$, we see that $V \cap D = \{s\}$. Thus, because $s\notin U$, we have that $U \cap V \cap D = \varnothing$.
%On the other hand, since $U \cap V$ is an open neighborhood of $x$ and $x\in\cl(D)$, we must have $U \cap V \cap D \ne \varnothing$.
%The obtained contradiction proves that $s \in \max\cl(D)$.
%
First observe that $(\down A \cap D) \cap B = \varnothing$. Otherwise there is $b \in (\down A \cap D) \cap B$, so $b \le a$ for some
$a \in A$.
Since $B$ is an upset of $D$, we have $a \in B$. Therefore, $a \in A \cap B$, which contradicts the assumption that $A,B$ are disjoint.
Next we show that $(\down A \cap S) \cap (\up B \cap S) = \varnothing$. If not, then there are
$b \in B$, $s \in S$, and $a \in A$ such that $b \le s \le a$. Since $a \in A \subseteq D$ and $B$ is an upset of $D$, we get $a \in B$,
which is false as $A \cap B = \varnothing$. Now, define $f : S \to [0,1]$ by $f(s) = 0$ if $s \in \down A \cap S$ and $f(s) = 1$ otherwise.
Clearly $f$ is order-preserving. Thus, since $X$ is the Nachbin compactification of $S$ (see Theorem~\ref{thm: Nachbin}), there is a
continuous order-preserving map $g : X \to [0,1]$ with $g|_S = f$. As $A \subseteq f^{-1}(0)$ and $B \subseteq f^{-1}(1)$, we then conclude that
$\cl(A) \subseteq g^{-1}(0)$ and $\cl(B) \subseteq g^{-1}(1)$, and hence $\cl(A) \cap \cl(B) = \varnothing$.
\end{proof}

In what follows we will make heavy use of the technique of nets and net convergence (see, e.g., \cite[Sec.~1.6]{Eng89}).
We recall that a \emph{net} in $X$ is a map $\mathfrak{n}$ from a directed set $\Gamma$ to $X$. We call a net $\mathfrak{n} : \Gamma \to X$
\emph{increasing} if $\gamma \le \delta$ implies $\mathfrak{n}(\gamma) \le \mathfrak{n}(\delta)$.

We call a subset $A$ of $X$ \emph{up-directed} if $A$ is a directed set with the induced order coming from $X$. If $A$ is up-directed,
then the inclusion function $A \to X$ is an increasing net in $X$.
%We will abuse terminology and call $A$ an increasing net.
Conversely, if $\mathfrak{n} : \Gamma \to X$ is an increasing net, then the image $\mathfrak{n}(\Gamma)$ is an up-directed subset of $X$.

\begin{lemma} \label{lem: limit is in Y}
Let $\mathfrak{n}$ be an increasing net in $Y$ converging to $x \in X$. Then $\mathfrak{n}(\Gamma) \subseteq \down x$ and $x \in Y$.
\end{lemma}

\begin{proof}
Let $A=\mathfrak{n}(\Gamma)$. We first show that $A \subseteq \down x$. If not, then there is $a \in A$ with $a \not\le x$.
By the Priestley separation axiom, there is a clopen upset $U$ of $X$ with $a \in U$ and $x \notin U$. Since $X \setminus U$ is an open
neighborhood of $x$ and $\mathfrak{n}$ is
a net converging to $x$, there is $\gamma \in \Gamma$ such that for all $\delta \ge \gamma$, we have $\mathfrak{n}(\delta) \in X\setminus U$. Because $\mathfrak{n}$ is increasing, there is $\delta$ with $a, \mathfrak{n}(\gamma) \le \mathfrak{n}(\delta)$. This implies $\mathfrak{n}(\delta) \in X \setminus U$, which is impossible since $\mathfrak{n}(\delta) \in U$ as $U$ is an upset and $a \in U$. The obtained
contradiction proves that $A \subseteq \down x$.

We next show that $x \in Y$. Let $V = X \setminus \down x$, an open upset of $X$. Since $X$ is an extremally order-disconnected Esakia space, $\cl(V)$ is a clopen
upset. Let $a \in A$. Then $a \le x$, so $a \notin V$, and so $\down a \cap V = \varnothing$ because $V$ is an upset. Since $a\in Y$,
we have $\down a$ is clopen, so $\down a \cap \cl(V) = \varnothing$. This implies $A \cap \cl(V) = \varnothing$, and so
$\cl(A) \cap \cl(V) = \varnothing$ as $\cl(V)$ is clopen. Since $x \in \cl(A)$, we conclude that $x \notin \cl(V)$, and hence
$\down x \cap \cl(V) = \varnothing$. This implies that $V = \cl(V)$, so $V$ is clopen. Therefore, $\down x$ is clopen. Thus, $x \in Y$.
\end{proof}

\begin{lemma} \label{lem: limit of an increasing net}
Let $A$ be an up-directed subset of $Y$. Viewing $A$ as a net, $A$ converges to a point $y \in Y$ with $A \subseteq \down y$.
\end{lemma}

\begin{proof}
Let $\mathfrak{n} : \Gamma \to X$ be an increasing net in $X$ with $\mathfrak{n}(\Gamma) = A$. Since $X$ is compact, $\mathfrak{n}$ has a
convergent subnet $\mathfrak{n}\circ \varphi$ for some order preserving map $\varphi:\Lambda \to \Gamma$ whose image is cofinal in $\Gamma$.
Set $B = \mathfrak{n}(\varphi(\Lambda))$. Let $y$ be the limit of $B$. By Lemma~\ref{lem: limit is in Y}, $B \subseteq \down y$ and
$y \in Y$. We show that $y$ is the supremum of $B$ in $X$. Suppose $x$ is an upper bound of $B$. If $y \not\le x$, then the Priestley
separation axiom yields a clopen downset $V$ of $X$ containing $x$ but not $y$. Since $B \subseteq \down x$, we have $B \subseteq V$,
so $B \cap (X \setminus V) = \varnothing$, which is impossible because $X \setminus V$ is a neighborhood of $y$ and $y$ is the limit of $B$.
Thus, $y \le x$, and so $y$ is the supremum of $B$. Let $a \in A$. Since $B$ is cofinal in $A$, there is $b \in B$ with $a \le b$.
Consequently, $A \subseteq \down y$, and so $y$ is also the supremum of $A$.

We show that $y$ is the limit of $A$. Suppose that $W$ is an
open neighborhood of $y$. Then there are clopen upsets $U,V$ with $y \in U\setminus V \subseteq W$. Since $X \setminus V$ is a clopen downset
and $y \in X \setminus V$, it follows that $A \subseteq \down y \subseteq X \setminus V$. Since $U$ is an open neighborhood of $y$, there is $\lambda \in \Lambda$
such that if $\delta \ge \lambda$, then $\mathfrak{n}(\varphi(\delta)) \in U$.  Suppose that $\gamma \in \Gamma$ with $\gamma \ge \varphi(\lambda)$.
Since $\mathfrak{n}$ is increasing, $\mathfrak{n}(\gamma) \ge \mathfrak{n}(\varphi(\lambda))$. We have $\mathfrak{n}(\varphi(\lambda)) \in U$
and $U$ is an upset, so $\mathfrak{n}(\gamma) \in U$. Therefore, $\mathfrak{n}(\gamma) \in W$ for each $\gamma \ge \varphi(\lambda)$. Thus,
$\mathfrak{n}$ converges to $y$.
\end{proof}

\begin{lemma} \label{lem: updirected}
Let $E$ be a clopen downset of $X$ such that $\max E \cap Y = \varnothing$. If $A$ is a nonempty upset of $E \cap S$, then $A$ is not up-directed.
\end{lemma}

\begin{proof}
Let $D = E \cap S$ and let $A$ be a nonempty upset of $D$. By Lemma~\ref{lem: downsets in S}(1), $A = U \cap D$ for some clopen upset $U$ of $X$. Suppose that $A$ is up-directed. Then the inclusion map $A \to X$ is an increasing net $\mathfrak{n}$.
Lemma~\ref{lem: limit of an increasing net} implies that $\mathfrak{n}$ converges to a point $y \in Y$ such that $A \subseteq \down y$. As
$\max E \cap Y = \varnothing$, there is $x \in E$ with $y < x$. Because $U$ is a clopen upset containing $A$ and $A \subseteq \down y$,
we see that $y \in U$, so $x \in U$. Consequently, $x \in U \cap E$. By Theorem~\ref{thm: compactification}, $S$ is dense in $X$. Therefore, $\cl(D) = \cl(E \cap S) = E \cap \cl(S) = E$ as $E$ is clopen. Thus, $D$ is dense in $E$, and hence $A$ is dense in $U \cap E$.
%$\cl(A) = \cl(U \cap D) = U \cap E$.
From $A \subseteq \down y$ it follows that $\cl(A) \subseteq \down y$. Therefore,
$x \in U \cap E = \cl(A) \subseteq \down y$. This is impossible since $y < x$. The obtained contradiction proves that $A$ is not up-directed.
\end{proof}

\begin{lemma} \label{lem: existence of an increasing net}
Suppose that $D$ is a downset of $S$ and there is $x \in X$ with $x \in \cl(\down x \cap D)$. Then there is an increasing net in $D$
which converges to $x$.
\end{lemma}

\begin{proof}
Let $x \in \cl(\down x \cap D)$. Then there is a net $\mathfrak{n} : \Gamma \to \down x \cap D$ converging to $x$.
We build an increasing net in $D$ converging to $x$. Let $A = \down x \cap D$. Then all the terms of the net are in $A$. We show that
$A$ is up-directed. Let $a,b \in A$. Since $\up a \cap D$ and $\up b \cap D$ are upsets of $D$, by Lemma~\ref{lem: downsets in S}(1) there
are clopen upsets $U,V$ of $X$ with
$U \cap D = \up a \cap D$ and $V \cap D = \up b \cap D$. Because $a,b \le x$, we see that $x \in U \cap V$. Therefore,
$U \cap V$ is an open neighborhood of $x$. Thus, $\mathfrak{n}(\gamma) \in U \cap V$ for some $\gamma$. This implies that $a,b \le \mathfrak{n}(\gamma)$. Since
$\mathfrak{n}(\gamma) \in A$, this shows that $A$ is up-directed. We may then view $A$ as an increasing net.

We show that $A$ converges to $x$. Let $W$
be an open neighborhood of $x$. Then there are clopen upsets $U,V$ of $X$ with $x \in U \setminus V \subseteq W$. As $X \setminus V$ is
an open downset containing $x$ and $A \subseteq \down x$, we have $A \subseteq X \setminus V$. There is $\delta \in \Gamma$ such that if
$\gamma \ge \delta$, then $\mathfrak{n}(\gamma) \in U$. If $a \in A$ with $\mathfrak{n}(\delta) \le a$, then $a \in U$ since $U$ is an upset. Consequently,
for each $a \in A$ with $\mathfrak{n}(\delta) \le a$, we have $a \in W$. This shows that the net $A$ converges to $x$. We have thus produced an
increasing net in $D$ converging to $x$.
\end{proof}

\begin{lemma} \label{lem: limits and covering}
Let $A,B\subseteq Y$ and $A \subseteq \down B$. If $x$ is a limit point of $A$, then there is a limit point $y$ of $B$ with $x \le y$.
\end{lemma}

\begin{proof}
Since $x$ is a limit point of $A$ there is a net $\mathfrak{n} : \Gamma \to A$ converging to $x$. For each $\gamma$ choose $b_\gamma \in B$
with $\mathfrak{n}(\gamma) \le b_\gamma$. Define a net $\mathfrak{m} : \Gamma \to Y$ by $\mathfrak{m}(\gamma) = b_\gamma$.
Since $X$ is compact, there is a subnet $\mathfrak{m}\circ \varphi$ of $\mathfrak{m}$ converging to some $y \in X$,
where $\varphi : \Lambda \to \Gamma$ is order preserving and its image is cofinal in $\Gamma$. Then $y$ is a limit point of $B$.
Because $\mathfrak{n}$ converges to $x$, the subnet $\mathfrak{n} \circ \varphi$ also converges to $x$ (see, e.g., \cite[Prop.~1.6.1]{Eng89}).

Suppose $x \not\le y$. By the
Priestley separation axiom, there is a clopen upset $U$ containing $x$ and missing $y$.
Since $x \in U$ there is $\lambda \in \Lambda$ such that for each $\delta \ge \lambda$, we have $\mathfrak{n}(\varphi(\delta)) \in U$.
As $U$ is an upset, $\mathfrak{m}(\varphi(\delta)) \in U$. Because
$\mathfrak{m} \circ \varphi$ converges to $y$ and $y \in X\setminus U$, there is $\lambda'$ such that for each $\delta \ge \lambda'$ we have
$\mathfrak{m}(\varphi(\delta)) \in X\setminus U$. Then, for any $\delta \ge \lambda, \lambda'$, we have
$\mathfrak{m}(\varphi(\delta)) \in U\cap (X\setminus U)$, which is impossible. Thus, $x \le y$.
\end{proof}

Let $\T$ be the infinite binary tree shown below.

\newcommand\rad{.11}
\newcommand\srad{.05}
\newcommand\brad{.25}
\newcommand\hgt{3}
\newcommand\len{15pt}
\begin{figure}[H]
\begin{tikzpicture}[scale=.3]
\foreach \i in {0,...,2}
{\draw[fill] (5.5 + 10*\i, 5.3*\hgt) circle[radius=\srad];
\draw[fill] (5.5 + 10*\i, 5.5*\hgt) circle[radius=\srad];
\draw[fill] (5.5 + 10*\i, 5.7*\hgt) circle[radius=\srad];}

\foreach \i in {0,..., 31}
\draw[fill] (\i, 5*\hgt) circle[radius=\rad];

\foreach \i in {0,...,15}
{\draw[fill] (0.5 + 2*\i, 4*\hgt) circle[radius=\rad];
\draw (0.5 + 2*\i, 4*\hgt) -- (2*\i, 5*\hgt);
\draw (0.5 + 2*\i, 4*\hgt) -- (2*\i+1, 5*\hgt);}

\foreach \i in {0,..., 7}
{\draw[fill] (1.5 + 4*\i, 3*\hgt) circle[radius=\rad];
\draw (1.5 + 4*\i, 3*\hgt) -- (.5 + 4*\i, 4*\hgt);
\draw (1.5 + 4*\i, 3*\hgt) -- (.5 + 4*\i+2, 4*\hgt);}

\foreach \i in {0,...,3}
{\filldraw (3.5 + 8*\i, 2*\hgt) circle[radius=\rad];
\draw (3.5 + 8*\i, 2*\hgt) -- (1.5 + 8*\i, 3*\hgt);
\draw (3.5 + 8*\i, 2*\hgt) -- (1.5 + 8*\i+4, 3*\hgt);}

\foreach \i in {0,...,1}
{\filldraw (7.5 + 16*\i, 1*\hgt) circle[radius=\rad];
\draw (7.5 + 16*\i, 1*\hgt) -- (3.5 + 16*\i, 2*\hgt);
\draw (7.5 + 16*\i, 1*\hgt) -- (3.5 + 16*\i+8, 2*\hgt);}

\filldraw (15.5,0) circle[radius=\rad];
\draw (7.5, 1*\hgt) -- (15.5, 0);
\draw (23.5, 1*\hgt) -- (15.5, 0);
\end{tikzpicture}
\caption{The infinite binary tree $\T$}
\end{figure}
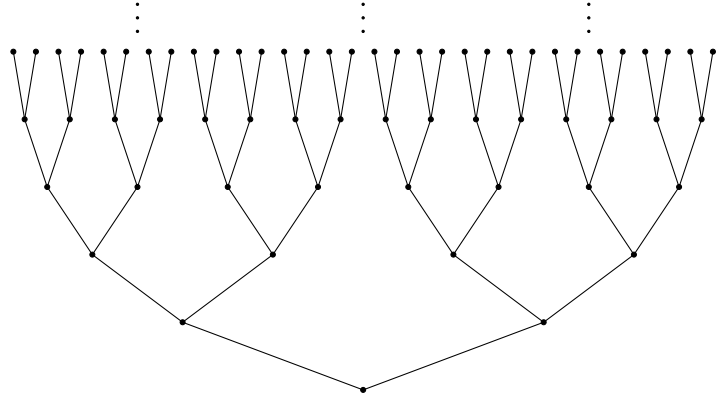
We think of $\T$ as built from \emph{combs} where a comb is depicted below.
\begin{figure}[H]
\begin{tikzpicture}[scale=.5]
\foreach \i in {0,...,5}
{\filldraw (2*\i,\i/5) circle[radius=\rad];
\filldraw (2*\i, 2+\i/5) circle[radius=\rad];
\draw (2*\i, \i/5) -- (2*\i, 2+\i/5);}
\foreach \i in {0,...,2}
{\filldraw (12 + \i, 1.2 +\i/10) circle[radius=\srad];
\filldraw (12 + \i, 3.2 + \i/10) circle[radius=\srad];}
\draw (0,0) -- (10, 1);
\end{tikzpicture}
\caption{A comb}
\end{figure}
Namely, we start with the root of $\T$ and build a comb with the ``round'' and ``square'' points drawn below. The round points form the ``spine'' of the comb and the square points the ``teeth'' of the comb. Then for each square point
we build a comb with the point as the root. Continuing this process yields $\T$.

\begin{figure}[H]
\begin{tikzpicture}[scale=.3]
\foreach \i in {0,...,2}
{\draw[fill] (5.5 + 10*\i, 5.3*\hgt) circle[radius=\srad];
\draw[fill] (5.5 + 10*\i, 5.5*\hgt) circle[radius=\srad];
\draw[fill] (5.5 + 10*\i, 5.7*\hgt) circle[radius=\srad];}

\foreach \i in {0,..., 31}
\draw[fill] (\i, 5*\hgt) circle[radius=\rad];

\foreach \i in {0,...,15}
{\draw[fill] (0.5 + 2*\i, 4*\hgt) circle[radius=\rad];
\draw (0.5 + 2*\i, 4*\hgt) -- (2*\i, 5*\hgt);
\draw (0.5 + 2*\i, 4*\hgt) -- (2*\i+1, 5*\hgt);}

\foreach \i in {0,..., 7}
{\draw[fill] (1.5 + 4*\i, 3*\hgt) circle[radius=\rad];
\draw (1.5 + 4*\i, 3*\hgt) -- (.5 + 4*\i, 4*\hgt);
\draw (1.5 + 4*\i, 3*\hgt) -- (.5 + 4*\i+2, 4*\hgt);}

\foreach \i in {0,...,3}
{\filldraw (3.5 + 8*\i, 2*\hgt) circle[radius=\rad];
\draw (3.5 + 8*\i, 2*\hgt) -- (1.5 + 8*\i, 3*\hgt);
\draw (3.5 + 8*\i, 2*\hgt) -- (1.5 + 8*\i+4, 3*\hgt);}

\foreach \i in {0,...,1}
{\filldraw (7.5 + 16*\i, 1*\hgt) circle[radius=\rad];
\draw (7.5 + 16*\i, 1*\hgt) -- (3.5 + 16*\i, 2*\hgt);
\draw (7.5 + 16*\i, 1*\hgt) -- (3.5 + 16*\i+8, 2*\hgt);}

\draw (7.5, 1*\hgt) -- (15.5, 0);
\draw (23.5, 1*\hgt) -- (15.5, 0);
\filldraw (15.5,0) circle[radius=\brad];
\filldraw (23.5,1*\hgt) circle[radius=\brad];
\filldraw ([xshift=-\len/2,yshift=-\len/2]7.5,1*\hgt) rectangle ++(\len,\len);
\filldraw (27.5,2*\hgt) circle[radius=\brad];
\filldraw ([xshift=-\len/2,yshift=-\len/2]19.5,2*\hgt) rectangle ++(\len,\len);
\filldraw (29.5,3*\hgt) circle[radius=\brad];
\filldraw ([xshift=-\len/2,yshift=-\len/2]25.5,3*\hgt) rectangle ++(\len,\len);
\filldraw (30.5,4*\hgt) circle[radius=\brad];
\filldraw ([xshift=-\len/2,yshift=-\len/2]28.5,4*\hgt) rectangle ++(\len,\len);
\filldraw (31,5*\hgt) circle[radius=\brad];
\filldraw ([xshift=-\len/2,yshift=-\len/2]30,5*\hgt) rectangle ++(\len,\len);
\end{tikzpicture}
\caption{A comb in $\T$}
\end{figure}

\begin{theorem}\label{thm:T_2}
$N(\Op \T)$ is not spatial.
\end{theorem}

\begin{proof}
Following Convention~\ref{convention}, we write $X$ for $X_{\Op\T}$ and $Y$ for $Y_{\Op\T}$. By Corollary~\ref{cor: spatial}, it is sufficient
to show that there is a clopen downset $E$ of $X$ such that $\max E\cap Y=\varnothing$; and we show that $\max X\cap Y=\varnothing$.
Since $\T$ is dense in $X$ by Theorem~\ref{thm: compactification}, we have that $X=\cl(\T)$. Suppose $y \in \max X \cap Y$.
%As $y \in E = \cl(D)$,
Then there is a net in $\T$ converging to $y$. Since $y \in Y$, we have $\down y$ is clopen.
Therefore, $\down y = \cl(\down y \cap \T)$ (see the proof of Lemma~\ref{lem: updirected}).
Consequently, by Lemma~\ref{lem: existence of an increasing net}, there is an increasing net $\mathfrak{n}:\Gamma\to \T$ converging to $y$.
Let $A=\mathfrak{n}(\Gamma)$.
Then $A$ is an up-directed subset of $\T$, so $A$ is a chain in $\T$.
Consider the comb that has $A$ as the spine. Let $B$ be the upset generated by the teeth of the comb.
Then $A \cap B = \varnothing$ and $A \subseteq \down B$. Therefore, by Lemma~\ref{lem: limits and covering}, there is a limit point $x$
of $B$ with $y \le x$.
%Since $B \subseteq D$, we have $x \in \cl(D) = E$.
By Lemma~\ref{lem: disjoint},
$\cl(A) \cap \cl(B) = \varnothing$. Since $y \in \cl(A)$ and $x \in \cl(B)$, we conclude that $y \ne x$. Therefore, $y < x$,
which is a contradiction to $y \in \max X$. Thus, $\max X \cap Y = \varnothing$, which shows that $N(\Op \T)$ is not spatial by
Corollary~\ref{cor: spatial}.
%$D=\T$. Then $\max D=\varnothing$. Suppose $A$ is an up-directed subset of $D$. Since $D$ is a tree, $A$ is a chain. We produce $B$ as follows.
%For each $a \in A$, there are two immediate successors of $a$. At least one of these successors is not in $\down A$ since $\down A$ is a chain.
%Let $B'$ be the set of all immediate successors of elements of $A$ which are not in $\down A$, and set $B = \up B'$. Then $B$ is an upset of $D$.
%To see that $A \cap B = \varnothing$, suppose instead that $a \in A \cap B$. Then $a \in \up B'$, so $b \le a$ for some $b \in B'$. But then
%$b \in \down A$, which is false from the definition of $B'$. The definition also shows that $A \subseteq \down B' = \down B$. Thus, $N(\Op \T)$
%is not spatial by Theorem~\ref{thm: non spatial}.
\end{proof}

\begin{lemma} \label{lem: subposet}
Let $S$ be a poset and $T$ a subposet of $S$. If $N(\Op T)$ is not spatial, then neither is $N(\Op S)$.
\end{lemma}

\begin{proof}
Since $T$ is a subposet of $S$, we see that $\Op T$ is a quotient of $\Op S$. It follows from the proof of \cite[Lem.~3.4]{Sim89} that
$N(\Op T)$ is isomorphic to an interval in $N(\Op S)$. Thus, spatiality of $N(\Op S)$ implies spatiality of $N(\Op T)$.
\end{proof}

We are ready to prove the main result of this paper.

\begin{theorem} \label{thm: tree}
Let $S$ be a poset. Then $N(\Op S)$ is not spatial iff $\T$ is isomorphic to a subposet of $S$.
\end{theorem}

\begin{proof}
First suppose that $\T$ is isomorphic to a subposet of $S$. Then Theorem~\ref{thm:T_2} and Lemma~\ref{lem: subposet} yield that $N(\Op S)$
is not spatial. Conversely, suppose that $N(\Op S)$ is not spatial. Then Corollary~\ref{cor: spatial} gives a nonempty clopen downset $E$ of $X$ with
$\max E \cap Y = \varnothing$. Let $D = E \cap S$. By Lemma~\ref{lem: updirected}, each nonempty upset of $D$ is not up-directed. In particular,
for each $x \in D$ the upset $\up x \cap D$ of $D$ is not up-directed.  Therefore, there are $y,z \in D$ with $x \le y, z$ but no $w \in D$ with
$y,z \le w$. We build a copy of $\T$ inside $D$  by first building a comb inside $D$.

Let $x_0 \in D$. Then there are $x_1, y_0 \in D$ with
$x_0 \le x_1, y_0$ such that nothing in $D$ is above both $x_1, y_0$. Repeating this construction, for each $n$ we produce $x_n \in D$ and $x_{n+1}, y_n  \in D$
with $x_n \le x_{n+1}, y_n$ such that nothing in $D$ is above both $x_{n+1}$ and $y_n$. We claim that $C = \{ x_n, y_n \mid n \in \mathbb{N} \}$
is a comb inside $D$. By construction, $x_0 < x_1 < \cdots$ is a chain in $D$, and $x_i \le y_i$ for each $i$. We need to show that
$\{ y_n \mid n \in \mathbb{N} \}$ is an antichain. Assume that there are $i \ne j$ with $y_i \le y_j$. First suppose that $i < j$.
The element $y_j$ is above both $y_i$ and $x_j$. Since $i < j$ and $\{x_n\}$ is an increasing chain, $y_j$ is above both $y_i$ and $x_{i+1}$.
This is impossible by construction. Next, suppose that $j < i$. Then $y_j \ge y_i \ge x_i \ge x_{j+1}$. This is false by construction.
Thus, $C$ is indeed a comb in $D$. By repeating this construction, we can build a comb in $D$ rooted at each $y_n$.

To see that
the resulting poset is $\T$, if $i < j$, then we show that the combs rooted at $y_i$ and $y_j$ are disjoint. Suppose otherwise. Then there is
$a \in D$ above both $y_i$ and $y_j$. Therefore, $a$ is above both $y_i$ and $x_j$. Since $i < j$, we have $x_{i+1} \le x_j \le a$. Thus, $a$
is above both $x_{i+1}$ and $y_i$, a contradiction. Hence, the combs above $y_i$ and $y_j$ are disjoint. The resulting subposet of $D$ is
then isomorphic to $\T$, completing the proof.
\end{proof}

\section{Consequences of the Main Theorem} \label{sec: main}

We conclude the paper by deriving some consequences of Theorem~\ref{thm: tree}. First we derive a characterization of when $N(\Op S)$ is
spatial for an arbitrary Alexandroff space. Let $S$ be an Alexandroff space, which we will view as a preordered set.
%For $T\subseteq S$, we call $x\in T$ \emph{quasi-maximal} if $x\le y$ and $y\in T$ imply $y\le x$. Let $\qmax T$ be the set of quasi-maximal
%points of $T$.
%\[
%\qmax D = \{ x \in D \mid x \le y \in D \implies y \le x\}.
%\]
Let $S_0$ be the skeleton ($T_0$-reflection) of $S$ and let $\rho:S\to S_0$ be the corresponding map sending $x\in S$ to $[x]\in S_0$.
Then $\rho^{-1}:\Op S_0\to \Op S$ is an isomorphism of frames (see the beginning of Section~3).
%It is easy to see that $\max\rho(T) = \rho(\qmax T)$ and $\qmax T = \rho^{-1}(\max\rho(T))$.
%$\qmax T = \{ x \in T \mid \rho(x) \in \max\rho(T) \}$ and
%We recall that the map sending $A$ to $[A]$ is a homeomorphism from $\Op S$ to $\Op S_0$.
%\textbf{The following proof would simplify a little if we give the map $S \to S_0$ a name.}

\begin{corollary}
For a preorder $S$ the following are equivalent.
\begin{enumerate}
\item $N(\Op S)$ is not spatial.
\item $\T$ is isomorphic to a subposet of $S_0$.
\item $\T$ embeds isomorphically into $S$.
\end{enumerate}
\end{corollary}

\begin{proof}
(1)$\Leftrightarrow$(2). Since $N(\Op S)$ is isomorphic to $N(\Op S_0)$, we have that $N(\Op S)$ is not spatial iff $N(\Op S_0)$ is not spatial.
Now apply Theorem~\ref{thm: tree}.

(2)$\Leftrightarrow$(3). Suppose that $\T$ is isomorphic to a subposet of $S_0$. We may identify $\T$ with its image in $S_0$. For each $t \in \T$ choose $s_t \in \rho^{-1}(t)$. Then sending $t$ to $s_t$ is the desired embedding of $\T$ into $S$. Conversely, suppose $\T$ embeds isomorphically into $S$. We may identify $\T$ with its image in $S$. Then $\rho(\T)$ is a subposet of $S_0$ isomorphic to $\T$.
\end{proof}

We next recall that a poset $S$ is \emph{noetherian} if $S$ has no infinite ascending chains. If $S$ is noetherian,
it is clear that $\T$ does not embed in $S$. Therefore, Theorem~\ref{thm: tree} yields that $N(\Op S)$ is spatial.
In fact, $S$ is a noetherian poset iff $S$, viewed as an Alexandroff space, is scattered. Therefore, Simmons's well-known
theorem \cite[Thm.~4.5]{Sim80} implies that $N(\Op S)$ is moreover boolean.

It is natural to ask whether there exist posets $S$ such that $N(\Op S)$ is spatial, but not boolean. Isbell's theorem \cite{Isb91},
that for a sober space $S$ the frame $N(\Op S)$ is spatial iff $S$ is weakly scattered, does not resolve this question since for a
poset $S$, the concepts of sober, weakly scattered, and scattered are all equivalent to $S$ being noetherian.
We show that Theorem~\ref{thm: tree} resolves this question in the positive by providing many such examples.
We recall that a poset $S$ is \emph{totally ordered} if it is a chain; that is, $s\le t$ or $t\le s$ for all $s,t\in S$.

\begin{corollary} \label{thm: chain1}
Let $S$ be a poset.
\begin{enumerate}
\item If $S$ has no infinite antichains, then $N(\Op S)$ is spatial.
\item If $S$ is totally ordered, then $N(\Op S)$ is spatial.
\end{enumerate}
\end{corollary}

\begin{proof}
(1) Suppose that $S$ has no infinite antichains. Since $\T$ has infinite antichains, $\T$ cannot be isomorphic to a subposet of $S$.
Thus, $N(\Op S)$ is spatial by Theorem~\ref{thm: tree}.

(2) If $S$ is totally ordered, then $S$ has no infinite antichains. Now apply (1).
\end{proof}

Consequently,
%by showing that for each totally ordered set $S$, the frame $N(\Op S)$ is spatial. Thus,
for each totally ordered set $S$ (or more generally for each poset $S$ with no infinite antichains), if $S$ is not noetherian, then
$N(\Op S)$ is spatial, but not boolean.

\begin{remark}
The converse of Corollary~\ref{thm: chain1}(1) is clearly false. For example, if $S$ is an infinite antichain, then $N\Op S$ is spatial
by Theorem~\ref{thm: tree} since $\T$ is not isomorphic to a subposet of $S$.
\end{remark}

By Theorem~\ref{thm: spatial} and Corollary~\ref{thm: chain1}(2), for a totally ordered set $S$, the frame $N(\Op S)$ is isomorphic to the frame
of opens of $(Y, \pi)$. As we pointed out at the end of Section~2, $(Y,\pi)$ is a zero-dimensional Haudorff space. As our final result, we determine when
$(Y, \pi)$ is compact, and hence a Stone space. Recall that a poset $S$ is \emph{artinian} if there are no infinite descending chains in $S$; equivalently, if every nonempty subset of $S$ has a minimum.

\begin{theorem} \label{thm: chain2}
For a totally ordered set $S$, the following are equivalent.
\begin{enumerate}
\item $Y=X$.
\item $S$ is artinian.
\item $(Y, \pi)$ is compact.
\end{enumerate}
\end{theorem}

\begin{proof}
(1) $\Rightarrow$ (2). Let $A$ be an infinite descending chain $a_0 > a_1 > \cdots$ in $S$. Then the closure of $A$ in $X$ is also a chain (\cite[Thm.~III.2.9]{Esa85}). Therefore, it has a unique minimum $x \in X$, and $x \notin A$ since $A$ is infinite. Thus, $\down x \cap A = \varnothing$. This implies that $\down x$ is not clopen, so $x \notin Y$. This is impossible since $Y = X$.
Thus, $S$ is artinian.

(2) $\Rightarrow$ (3). Since $X$ is compact, it is sufficient to show that $X=Y$. Because $S$ is a chain and $S$ is dense in $X$, we have that $X$ 
is a chain (see \cite[Thm.~III.2.9]{Esa85}). Let $x \in X$. If $x$ is the maximum of $X$, then $\down x = X$, so $\down x$
is clopen, and hence $x \in Y$. Suppose not. Since $x$ is not the maximum of $X$, the set $\up x \setminus \{x\}$ is nonempty. Because $X$ is a 
chain, $\up x \setminus \{x\} = X \setminus \down x$. Therefore, $\up x \setminus \{x\}$ is a nonempty open subset of $X$. Thus, 
$(\up x \setminus \{x\}) \cap S \ne \varnothing$. Since $S$ is artinian, $(\up x \setminus \{x\}) \cap S$ has
a minimum $s$. Because $X$ is a chain, $(x,s) = X \setminus (\down x \cup \up s)$. Therefore, the interval $(x,s)$ is open. 
It misses $S$, so $(x,s) = \varnothing$ since $S$ is dense in $X$.
Therefore, $\down x = X \setminus \up s$, so $\down x$ is open. Thus, $\down x$ is clopen, yielding $x \in Y$.

(3) $\Rightarrow$ (1). If $Y$ is compact, then $Y = X$ since $Y$ is a closed dense subset of $X$.
\end{proof}

\begin{remark}
If $S$ is an artinian totally ordered set, then $X$ is in fact a compact ordinal. To see this, as we pointed out in the proof of
Theorem~\ref{thm: chain2}, $X$ is a chain. In addition, the topology on $X$ is the interval topology (see, \cite[Thm.~III.2.17]{Esa85}). Thus, it is sufficient to show that $X$ is artinian. If not,
then there is an infinite descending chain $x_0 > x_1 > \cdots$ in $X$. By Theorem~\ref{thm: chain2},  $X = Y$, so each $\down x_i$
is clopen. Therefore, $\down x_i \setminus \down x_{i+1}$ is a nonempty clopen set in $X$. Since $S$ is dense in $X$, there is
$s_i \in S$ with $s_i \in \down x_i \setminus \down x_{i+1}$. Because $X$ is a chain, $x_i \ge s_i > x_{i+1} \ge s_{i+1}$ for each $i$.
This implies that $s_0 > s_1 > \cdots$ is an infinite descending chain in $S$, which is impossible since $S$ is artinian.
Thus, $X$ is artinian, and hence is a compact ordinal.
\end{remark}

\begin{remark}
Theorem~\ref{thm: chain2} is not true in general. Let $S$ be an infinite antichain. Then $S$ is artinian. On the other hand, $X$ is
homeomorphic to the Stone-\v{C}ech compactification of $S$, and hence $Y=S\ne X$.
\end{remark}

\def\cprime{$'$}
\providecommand{\bysame}{\leavevmode\hbox to3em{\hrulefill}\thinspace}
\providecommand{\MR}{\relax\ifhmode\unskip\space\fi MR }
\providecommand{\MRhref}[2]{%
  \href{http://www.ams.org/mathscinet-getitem?mr=#1}{#2}
}
\providecommand{\href}[2]{#2}

\bigskip

Department of Mathematical Sciences, University of Texas at El Paso, El Paso, TX 79968, fco.avila.mat@gmail.com

Department of Mathematical Sciences, New Mexico State University, Las Cruces NM 88003, guram@nmsu.edu

Department of Mathematical Sciences, New Mexico State University, Las Cruces NM 88003, pmorandi@nmsu.edu

Departament of Mathematics, University of Guadalajara, Blvd.\ Marcelino Garc{\'i}a Barrag\'{a}n, 44430, Guadalajara, Jalisco, M\'{e}xico,
angelus31415@gmail.com

\end{document}